\documentclass[11pt,twoside]{amsart}
\usepackage[dvips]{graphicx}
\usepackage{amssymb}
\usepackage{latexsym}
\usepackage{xypic}
\usepackage{multirow}

\newtheorem{theorem}{Theorem}[section]
\newtheorem{lemma}[theorem]{Lemma}
\newtheorem{proposition}[theorem]{Proposition}

\newtheorem{corollary}[theorem]{Corollary}

\theoremstyle{definition}
\newtheorem{definition}[theorem]{Definition}

\theoremstyle{remark}

\newenvironment{definition-proposition}{\begin{def-prop} \em}{\end{def-prop}}

\def\P{\mathbb P_{\mathbb C}}
\def\H{\mathbb H}

\numberwithin{equation}{section}

\newcommand{\C}{\mathbb{C}}

\begin{document}

\title{  \sc{ Estimates of the number of lines lying in the limit set for subgroups of $PSL(3,\Bbb{C})$}}
\author{W. Barrera, A. Cano \& J. P. Navarrete }
\address{Waldemar Barrera: Universidad Aut\'onoma de Yucat\'an Facultad de Matem\'aticas, Anillo Perif\'erico
Norte Tablaje Cat 13615 Chuburn\'a Hidalgo, M\'erida Yucat\'an. M\'exico.\\
Angel Cano:   Instituto de Matem\'atica Pura e Aplicada (IMPA), Rio de Janeiro, Brazil.\\
Juan Pablo Navarrete:  Universidad Aut\'onoma de Yucat\'an Facultad de Matem\'aticas, Anillo Perif\'erico
Norte Tablaje Cat 13615 Chuburn\'a Hidalgo, M\'erida Yucat\'an. M\'exico.\\ }
\email{bvargas@uady.mx, angel@impa.br, jp.navarrete@uady.mx}

\thanks{Research partially supported by grants from CNPq}




\begin{abstract}
Given a discret subgroup $\Gamma\subset  PSL(3,\C)$,  we determine the number of complex lines and complex lines in general position lying in the complement of:  maximal  regions on which   $\Gamma$ acts properly discontinuously, the Kularni's limit set of  $\Gamma$ and  the equicontinuity set of $\Gamma$. We also provide sufficient conditions to ensure that the equicontinuity region agrees with the Kulkarni's discontinuity region and is the  largest set where the group acts properly discontinuously and we provide a description of he respective limit set in terms of the elements of the group.   
\end{abstract}

\maketitle

\section*{ Introduction}
Given a discrete group $\Gamma\subset PSL(n+1,\C)$, $n\neq 3$, is known that $\Gamma$ may have various maximal sets on which $\Gamma$ acts properly discontinuously, see \cite{cns}. And is not clear how the various extensions of the definition of discontinuity regions, see \cite{cns, cs, kul}, are related. This article is a continuation of our previous articles, see \cite{bcn1, bcn2}, on which have tried to solve the after told problem in the case $n=2$. Our argument, in the present article, is based on the existence of lines in the complement of discontinuity regions, see \cite{cano}, the generalization of the Montel's lemma for Normal families, see \cite{bcn1, kobayashi}   and the following estimations, see their respective proofs bellow:

\begin{theorem} \label{t:main1}
Let $\Gamma\subset PSL(3,\Bbb{C})$ be a discrete subgroup and $\Omega\subset \Bbb{P}^2_{\Bbb{C}}$ be either $Eq(\Gamma)$, $\Omega(\Gamma)$ or a maximal region where $\Gamma$ acts properly  discontinuously, then $Li(\Omega)=1,2,3$ or $\infty$.
\end{theorem}

\begin{theorem} \label{t:main2}
Let $\Gamma\subset PSL(3,\Bbb{C})$ be a discrete subgroup and $\Omega\subset \Bbb{P}^2_{\Bbb{C}}$ be either $Eq(\Gamma)$, $\Omega(\Gamma)$ or a maximal region where $\Gamma$ acts properly  discontinuously, then $LiG(\Omega)=1,2,3,4$ or $\infty$.
\end{theorem}

By meas of this  results we are able to show:
\begin{theorem} \label{t:main3}
Let $\Gamma \subset PSL(3, \Bbb{C})$ be  a discrete group. If the number of lines  of  each maximal  set $\Omega\subset \P^2$ on which  $\Gamma$ acts properly discontinuously  is at least 3,  then the Kulkarni's discontinuity region
$\Omega(\Gamma)$ agrees with the
equicontinuity set of the group $\Gamma$, is the largest open set where $\Gamma$ acts properly discontinuously and   $\Lambda(\Gamma)$ is the union of complex lines. Moreover, it holds that:
\[
\Lambda(\Gamma)=\overline{\bigcup_{\gamma\in \Gamma}
\Lambda(\gamma)}.
\]
 \end{theorem}
Such theorem extent our previous results in  \cite{bcn1}.\\

This article is  organized as follows: in section \ref{s:prel} we introduce some terms  and  notations which will be used along the text. In section  \ref{s:conlin} we prove Theorems \ref{t:main1} and \ref{t:main2}.  In section we prove Theorem \ref{t:main3}. Finally in section \ref{s:ex}, we provide examples which show that our estimates are optimal.

\section{Preliminaries and Notations} \label{s:prel}

\subsection{Projective Geometry}
We recall that the complex projective plane
$\mathbb{P}^2_{\mathbb{C}}$ is $$\Bbb{P}^2_\Bbb{C}:=(\mathbb{C}^{3}\setminus \{0\})/\mathbb{C}^*
,$$
 where $\mathbb{C}^*$ acts on $\mathbb{C}^3\setminus\{0\}$ by the usual scalar
 multiplication. Let $[\mbox{ }]:\mathbb{C}^{3}\setminus\{0\}\rightarrow
\mathbb{P}^{2}_{\mathbb{C}}$ be   the quotient map. A set  $\ell\subset \mathbb{P}^2_{\mathbb{C}}$ is said to be a
complex line if $[\ell]^{-1}\cup \{0\}$ is a complex linear
subspace of dimension $2$. Given  $p,q\in
\mathbb{P}^2_{\mathbb{C}}$ distinct points,     there is a unique
complex line passing through  $p$ and $q$, such line will be
denoted by $\overleftrightarrow{p,q}$. The set of all the complex lines contained in $\P^2$,  denoted $Gr(\P^2)$, endowed  with the topology of the Hausdorff convergence, turns out to be homeomorphic $\P^2$. \\

Consider the action of $\mathbb{Z}_{3}$ (viewed as the cubic roots of
the unity) on  $SL(3,\mathbb{C})$ given by the usual scalar
multiplication, then
$$PSL(3,\mathbb{C})=SL(3,\mathbb{C})/\mathbb{Z}_{3}$$ is a Lie group
whose elements are called projective transformations.  Let
$[[\mbox{  }]]:SL(3,\mathbb{C})\rightarrow PSL(3,\mathbb{C})$ be   the
quotient map,   $\gamma\in PSL(3,\mathbb{C})$ and  $\widetilde\gamma\in
GL(3,\mathbb{C})$, we will say that  $\tilde\gamma$  is a  lift of
$\gamma$ if there is a cubic root $\tau$ of  $Det(\gamma)$ such that   $[[\tau \widetilde\gamma]]=\gamma$, also, we will use the notation $(\gamma_{ij})$ to denote elements  in $SL(3,\Bbb{C})$. One can show that
$ PSL(3,\mathbb{C})$ is a Lie group  that acts  transitively,
effectively and by biholomorphisms  on $\mathbb{P}^2_{\mathbb{C}}$
by $[[\gamma]]([w])=[\gamma(w)]$, where $w\in
\mathbb{C}^3\setminus\{0\}$ and    $\gamma\in SL_3(\mathbb{C})$.\\
 
Now define the space of the pseudo-projective maps:
 $$SP(3,\C)=(M(3,\C)\setminus \{0\})/\mathbb{C}^*,$$
 where $\mathbb{C}^*$ acts on $M(3,\C)\setminus \{0\}$ by the usual scalar
 multiplication. In the sequel $[[\mbox{ }]]:M(3,\C)\setminus \{0\}\rightarrow
SP(3,\C)$  will denote  the quotient map  and we will consider   $SP(3,\C)$ endowed  with the quotient topology.  Finally,   given $\gamma\in SP(3,\C)$  we define its kernel by:
$$Ker(\gamma)=[Ker(\tilde \gamma)\setminus\{0\}],$$
where $\tilde \gamma\in M(3,\C)$ is a lift of $\gamma$. Clearly $PSL(3,\Bbb{C})\subset SP(3,\Bbb{C})$ and $\gamma\in PSL(3,\C)$ if and only if $Ker(\gamma)=\emptyset$.

\subsection{ Discontinuous Actions on $\Bbb{P}^2_{\Bbb{C}}$}
\begin{definition}
 Let $\Gamma\subset PSL(3,\C)$ be a discrete group. An open  non-empty  set $\Omega\subset \P^2$ is said to be a region where  $\Gamma$ acts properly discontinuously  if $\Gamma \Omega=\Omega$ and  for each compact set  $K\subset \Omega$ the set $\{\gamma\in \Gamma\vert \gamma(K)\cap K\}$ is finite.  Set  
 \[
 Dis(\Gamma)=\{\Omega\subset \P^2: \Omega\neq \emptyset\, \& \,\Gamma \textrm{ acts properly discontinuously on } \Omega\}.
 \]
  The group $\Gamma$ is said to be Kleinian if $Dis(\Gamma)\neq \emptyset$.

\end{definition}

\subsubsection{The equicontinuity set }

Recall that the {\it equicontinuity set} for a family $\mathcal{F}$ of
endomorphisms of $\mathbb {P}^2_\mathbb {C}$, denoted
$Eq(\mathcal{F})$, is defined to be the set of points $z\in
\mathbb{P}^2_\mathbb{C}$ for which there is an open neighborhood $U$
of  $z$   such that $\{ f\vert_U : f \in \mathcal{F}\}$ is a normal family.

\begin{definition}
Let $\Gamma\subset PSL(3,\C) $ be a discrete group, then we define:
$$Lim(\Gamma)=\{g\in SP(3,\Bbb{C}):g \textrm{ is a cluster point of } \Gamma\}.$$
\end{definition}

\begin{theorem}[See \cite{bcn1}]\label{dissubeq}
If $\Gamma \subset PSL(3, \C)$ is a discrete group, then:
\begin{enumerate}
\item The equicontinuity region  of $\Gamma$ is a region where $\Gamma$ acts properly discontinuously;
\item The equicontinuity region is described as follows:
$$Eq(\Gamma)=\P^2\setminus \left ( \bigcup_{ \gamma\in Lim(\Gamma)} Ker(\gamma) \right);$$
\item If   $U$ is a $\Gamma$-invariant open subset of $\P^2$
such that $\P^2  \setminus U$ contains at least three complex lines
in general position, then $U \subset Eq(\Gamma).$

\end{enumerate}
\end{theorem}

\subsubsection{The Discontinuity Region in the Kulkarni's sense}
Given  $\Gamma\subset PSL(3,\Bbb{C})$  a discrete group we define following Kulkarni, see  \cite{kul}: the set 
 $L_0(\Gamma)$  as the closure  of  the points in
$\mathbb{P}^2_{\mathbb{C}}$ with infinite isotropy group. The set $L_1(\Gamma)$ as the closure of the set  of cluster points  of
$\Gamma z$  where  $z$ runs  over  $\mathbb{P}^2_{\mathbb{C}}\setminus
L_0(\Gamma)$.
 The set   $L_2(\Gamma)$ as  the closure of cluster  points of $\Gamma
K$  where $K$ runs  over all  the compact sets in
$\mathbb{P}^2_{\mathbb{C}}\setminus (L_0(\Gamma) \cup L_1(\Gamma))$. The  \textit{Limit Set in the sense of Kulkarni} for $\Gamma$  is
defined as:  $$\Lambda (\Gamma) = L_0(\Gamma) \cup
L_1(\Gamma) \cup L_2(\Gamma).$$  The \textit{Discontinuity
Region in the sense of Kulkarni} of $\Gamma$ is defined as:
$$\Omega(\Gamma) = \mathbb{P}^2_{\mathbb{C}}\setminus
\Lambda(\Gamma).$$
Clearly the discontinuity region $\Omega(\Gamma)$ contains $Eq(\Gamma)$, see \cite{bcn1}, and is a region where $\Gamma$ acts properly discontinuously, see \cite{kul}.

\section{Counting Lines} \label{s:conlin}

\begin{definition}
Let $\Omega\subset \P^2$  be a non-empty open set. Let us define:
\begin{enumerate}
\item The lines in general position outside $\Omega$ as:
\begin{small}
\[
LG(\Omega)=\left \{\mathcal{L} \subset Gr_1(\P^2)\vert \textrm{The lines in } \mathcal{L} \textrm{ are in general position } \&\,  \bigcup\mathcal{L}\subset  \P^2\setminus \Omega \right \}; 
\]
\end{small}
\item The number of lines in general position outside $\Omega$ as:
\[
LiG(\Omega)=max(\{card(\mathcal{L}): \mathcal{L}\in  LG(\Omega)\}), 
\]
where $card(C)$ denotes the number of elements contained in $C$.
\item Given $\mathcal{L}\in 
LG(\Omega)$ and $v\in \bigcup \mathcal{L}$, we will say that  $v$ is a vertex for  $\mathcal{L}$ if there are $\ell_1,\ell_2\in \mathcal{L}$ distinct lines and   an infinite set $\mathcal{C}\subset  Gr_1(\P^2)$  such that $\ell_1\cap\ell_2 \cap(\bigcap \mathcal{C})=\{v\}$ and 
$\bigcup\mathcal{C}\subset \P^2\setminus \Omega$.
\end{enumerate}   
\end{definition}

\begin{corollary}
Let $\Gamma\subset PSL(3,\C)$ be a discrete group and  $\Omega\subset \P^2$ be a maximal  discontinuity region for the action of $\Gamma$ if $LiG(\Omega(\Gamma))\leq 2$, then $LiG(\Omega)\leq 2$.
\end{corollary}
\begin{proof}
On the contrary, let us assume that there is a maximal region $\Omega\subset \Bbb{P}^2_\Bbb{C}$ where $\Gamma$ acts properly discontinuously and $LiG(\Omega)\geq 3$, thus $\Omega\subset Eq(\Gamma)\subset \Omega(\Gamma)$. Thus $\Omega=\Omega(\Gamma)$, which is a contradiction.
\end{proof}

\begin{definition}
Let $\Gamma\subset PSL(3,\Bbb{C})$ be a group and $p\in \Bbb{P}^2_{\Bbb{C}}$ such that $\Gamma p=p$, then define
$\Pi=\Pi_{p,\ell,\Gamma}:\Gamma\longrightarrow  Bihol(\ell) $
given by $\Pi(g)(x)=\pi(g(x))$ where
$\pi=\pi_{p,\ell}:\mathbb{P}^2_{\mathbb{C}}- \{p\}\longrightarrow
\ell$ is given by $$\pi(x)=\overleftrightarrow{x,p}\cap \ell.$$
 Also, if $p=e_1$ define 
 $\lambda_{\infty,\Gamma}:\Gamma\rightarrow \C$ by 
\[
\lambda_{\infty,\Gamma}([\gamma])=\gamma_{11}^3.
\]
\end{definition}
Clearly,  $\Pi_{p,\ell,\Gamma}$ and $\lambda_{\infty,\Gamma}$ are  well defined group morphisms.

\begin{lemma} \label{l:pfix}
\label{i:12f} Let $\Gamma\subset PSL(3,\C)$ be a discrete group and  $\Omega\subset \P^2$ be an open  $\Gamma$-invariant set such that   $LiG(\Omega)=2$, then  there is a point $p\in \Bbb{P}^2_\Bbb{C}$ such that $\Gamma p=p.$
\end{lemma}
\begin{proof}
Let $\mathcal{L}=\{\ell\in Gr(\P^2): \ell\in  \P^2\setminus\Omega\}$. Since $LiG(\Omega)=2$, if follows that $\bigcap\mathcal{L}$ 
contains exactly one point, namely $p_0$.  Clearly $\Gamma \mathcal{L}=\mathcal{L}$ and in consequence $\Gamma p=p$.
\end{proof}

\begin{lemma}
Let $\Gamma\subset PSL(3,\C)$ be a discrete group and  $\Omega\subset \P^2$ be an open  $\Gamma$-invariant set such that $L_0(\Gamma)\subset \Bbb{P}^2_{\Bbb{C}}\setminus \Omega$. If   $LiG(\Omega)=2$ and $3 \leq Li(\Omega)<\infty$,  then:
\begin{enumerate}
\item \label{i:22f} If $\ell\in Gr(\Bbb{P}^2_\Bbb{C})$ does not contain p, then $\Pi_{p,\ell}(\Gamma)$ is finite.
\item \label{i:32f} It holds that  $Ker(\Pi)$ is a normal subgroup of finite index.
\item \label{i:42f} 
By conjugating with a projective transformation, if it is necessary, it follows that 
each $\gamma\in Ker(\Pi)$  with infinite order has a lift $\widetilde \gamma\in SL(3,\C)$ given by:
\[
\widetilde\gamma=
\left (
 \begin{array}{lll}
  1 & b & c\\
  0 & 1 & 0\\
  0 & 0 & 1\\
 \end{array}
\right).
\]  
\item \label{i:52f} By conjugating with a projective transformation, if it is necessary, it follows that each $\gamma\in Ker(\Pi)$  with infinite order has a lift $\widetilde \gamma \in SL(3,\C)$ given by:
\[
\widetilde\gamma=
\left (
 \begin{array}{lll}
  1 & b & 0\\
  0 & 1 & 0\\
  0 & 0 & 1\\
 \end{array}
\right).
\]  
\item \label{i:62f} By conjugating with a projective transformation, if it is necessary, it follows that the group  $\lambda_{\infty,\Gamma}(Ker(\Pi))$ is  finite.
\item \label{i:72f} There is $\ell\in Gr(\Bbb{P}^2_{\Bbb{C}})$ such that $Eq(\Gamma)=\P^2\setminus \ell.$
\end{enumerate}
\end{lemma}
\begin{proof}
Set  $\mathcal{L}=\{\ell\in Gr(\P^2): \ell\in  \P^2\setminus\Omega\}$ and $n_0=Card(\mathcal{L})$.\\

Let us prove (\ref{i:22f}). Since  
$\Gamma \bigcup\mathcal{L}=\bigcup\mathcal{L}$ it follows that $$\Pi(\Gamma)\pi\left (\bigcup\mathcal{L}\setminus \{p_0\}\right )= \pi\left (\bigcup\mathcal{L}\setminus \{p_0\}\right ).$$ 
In consequence $\pi\left (\bigcup\mathcal{L}\setminus \{p_0\}\right)$ is a $\Pi(\Gamma)$-invariant set whose cardinality is $n_0$. Thus 
$$\Gamma_0=\bigcap_{x\in\pi\left (\bigcup\mathcal{L}\setminus \{p_0\}\right) } Isot(x, \Pi(\Gamma)),$$ is a normal subgroup of $\Pi(\Gamma)$ with finite index. Moreover, every element in $\pi\left (\bigcup\mathcal{L}\setminus \{p_0\}\right ) $ is fixed by $\Gamma_0$. Since $\pi\left(\bigcup\mathcal{L}\setminus \{p_0\}\right) $  contains more than 3  elements we conclude that $\Gamma_0$ consist exactly on the identity. Therefore $\Pi(\Gamma)$ is finite.\\

Let us prove (\ref{i:32f}). 
From the previous claim it follows that $Ker(\Pi)$ is a normal subgroup of $\Gamma$ with finite index.\\

Let us prove (\ref{i:42f}). 
Let $\gamma\in Ker(\Pi)$ and $\widetilde\gamma\in SL(3,\C)$ a lift of $\gamma$, thus 
\[
\widetilde\gamma=
\left (
 \begin{array}{lll}
  a^{-2} & b & c\\
  0 & a & 0\\
  0 & 0 & a\\
 \end{array}
\right)
\]
where $a\in \C^*$ and $b,c\in \C$. if $\widetilde \gamma$ is diagonalizable,  there are  $v_1,v_2\in \C^3$ such that $\{e_1,v_1,v_2\}$ is an eigenbasis of  $\widetilde \gamma$ whose respective eigenvalues are $\{a^{-2},a,a\}$. In consequence  $\ell=\overleftrightarrow{v_1,v_2}\subset Fix(\gamma)\subset \P^2\subset \Omega$ and $e_1\notin \ell$, which is a contradiction. In consequence $\widetilde\gamma$ is non diagonalizable and the claim follows from the Jordan's normal form theorem.\\

By the previous claim we can  assume that    there is  an element $\gamma_0\in Ker(\Pi)$, such that $\gamma_0$ has a lift $\widetilde\gamma_0\in SL(3,\C)$ which is given by:

\[
\widetilde\gamma_0=
\left (
 \begin{array}{lll}
  1 & 1 & 0\\
  0 & 1 & 0\\
  0 & 0 & 1\\
 \end{array}
\right)
\]  

Let us prove (\ref{i:52f}). If this is not the case,  there is  $\gamma_1\in Ker(\Pi)$ which has a lift $\widetilde\gamma_1\in SL(3,\C)$ given by:
\[
\widetilde\gamma_1=
\left (
 \begin{array}{lll}
  1 & a & b\\
  0 & 1 & 0\\
  0 & 0 & 1\\
 \end{array}
\right).
\]
An inductive argument shows:
\[
\widetilde \gamma_0^m\widetilde\gamma_1=
\left (
 \begin{array}{lll}
  1 & a+n & b\\
  0 & 1 & 0\\
  0 & 0 & 1\\
 \end{array}
\right).
\]
A straightforward calculation shows that 
\[
\ell_m=\overleftrightarrow{[e_1],[0:-b:b+n]}\subset Fix(\gamma_0^m\gamma_1)
\]
and $\ell_m\neq \ell_n$, whenever $n\neq m$. That is $\ell_m\in \mathcal{L}$, which is a contradiction. Which concludes the proof of the claim.\\

Let us prove (\ref{i:62f}). On the contrary, let us assume that $\lambda_{\infty,\Gamma}(Ker(\Pi))$ is infinite, thus there is  a sequence $\lambda_m^3\subset \lambda_{\infty,\Gamma}(Ker(\Pi))$ of distinct elements
such that $\lambda_m^{1/2}\rightarrow 1$. For each $m\in \Bbb{N}$ let $\gamma_m\in Ker(\Pi)$ be such that $\lambda_{\infty,\Gamma} (\gamma_m)=\lambda_m^3$. Therefore  $\gamma_m$ has a lift $\widetilde \gamma_m\in SL(3,\C)$ given by:
\[
\widetilde\gamma=
\left (
 \begin{array}{lll}
  \lambda_m & a_m              & b_m\\
  0         & \lambda_m^{-1/2} & 0\\
  0         & 0                & \lambda_m^{-1/2}\\
 \end{array}
\right).
\] 
where $a_m,b_m\in \C$. Now a simple calculation shows:
\[
\widetilde\gamma_m^{-1}\widetilde\gamma_0\widetilde\gamma_m=
\left (
 \begin{array}{lll}
  1 & \lambda_m^{-3/2} & 0\\
  0 & 1                & 0\\
  0 & 0                & 1\\
 \end{array}
\right) \xymatrix{
\ar[r]_{m \rightarrow  \infty}&} 
\left (
 \begin{array}{lll}
  1 & 1 & 0\\
  0 & 1 & 0\\
  0 & 0 & 1\\
 \end{array}
\right).
\]
Which is a contradiction, since $\Gamma$ is discrete.\\

Let us prove (\ref{i:72f}). By the previous claims, it follows that each element $\gamma\in Ker(\lambda_{\infty,Ker(\Pi)})$ has a lift $\widetilde \gamma\in SL(3,\C)$ which is given by:
\[
\widetilde\gamma=
\left (
 \begin{array}{lll}
  1 & a & 0\\
  0 & 1 & 0\\
  0 & 0 & 1\\
 \end{array}
\right).
\]
It follows that $Eq( Ker(\lambda_{\infty,Ker(\Pi)}))=\P^2\setminus \overleftrightarrow{e_1,e_3}.$ Which concludes the proof of the claim.
\end{proof}

 \begin{definition} Let $\Gamma\subset PSL(3,\Bbb{C})$ be a discrete group, let us define 
 $C(\Gamma)$ as the set:
$$C(\Gamma) = \overline{ \bigcup _{\gamma \in \Gamma} \Lambda(\gamma)}.$$
\end{definition}

Clearly  $C(\Gamma)\subset \Bbb{P}^2_{\Bbb{C}}\setminus Eq(\Gamma)$  is a non empty, $\Gamma-$invariant closed set, see \cite{Nav1}, \cite{cano}.

\begin{proposition}
Let $\Gamma\subset PSL(3,\Bbb{C})$ be a discrete subgroup and $\Omega\subset \Bbb{P}^2_{\Bbb{C}}$ be either $Eq(\Gamma)$, $\Omega(\Gamma)$, $\Bbb{P}^2_{\Bbb{C}}\setminus C(\Gamma)$ or a maximal region where $\Gamma$ acts properly  discontinuously. If $LiG(\Omega)=2$ and $Li(\Omega)<\infty$, then $LiG(\Omega)=2$.
\end{proposition}
\begin{proof}
If $\Omega$ is either  $Eq(\Gamma)$,  $\Omega(\Gamma)$ or $\Bbb{P}^2_{\Bbb{C}}\setminus C(\Gamma)$ and  the assertion does not follows, then previous lemma yields that $\Bbb{P}^2_{\Bbb{C}}\setminus(\Omega)$ is a complex line, which is a contradiction.\\

Finally, if  $\Omega$ is a maximal region where $\Gamma$ acts  properly discontinuously and $Li(\Omega)>2$, then  parts (\ref{i:72f}) and (\ref{i:82f}) of previous lemma yields $ L_{0}(\Gamma)=\Bbb{P}^2_{\Bbb{C}}\setminus Eq(\Gamma)$ is a complex line. 	Thus $\Omega\subset Eq(\Gamma) $, which is a contradiction.
\end{proof}

\begin{proposition}
Let $\Gamma\subset PSL(3,\C)$ be a discrete group and  $\Omega\subset \P^2$ be an open  $\Gamma$-invariant set such that $L_0(\Gamma)\subset \Bbb{P}^2_{\Bbb{C}}\setminus \Omega$. If   $LiG(\Omega)=3$ and $3 < Li(\Omega)<\infty$,  then
 there is $\ell\in Gr(\Bbb{P}^2_{\Bbb{C}})$ and $p\in \P^2\setminus \ell$ such that $Eq(\Gamma)=\P^2\setminus (\ell\cup \{p\})=\P^2\setminus L_0(\Gamma).$
\end{proposition}

\begin{proof}
On the contrary, let us assume that $n_0=Li(\Omega)\geq 3$. Let $\mathcal{L}=\{\ell\in Gr(\Bbb{P}^2_{\Bbb{C}})\vert \ell\in \P^2\setminus \Omega\}$. Let $\Gamma_0=\bigcap_{\ell\in \mathcal{L}} Isot(\ell,\mathcal)$. Since $\mathcal{L}$ is a finite set, we conclude that $\Gamma_0$ is a normal subgroup of $\Gamma$ with finite index. Since $LiG(\Omega)=3$, and after conjugating by a  projective transformation if it is necessary, we can assume that $X =\overleftrightarrow{[e_1],[e_2]},\, Y=\overleftrightarrow{[e_3],[e_2]},\,Z=\overleftrightarrow{[e_1],[e_3]}\in \mathcal{L}$. Thus  every element  $\gamma\in \tau\Gamma_0\tau^{-1}$   has a lift in $\widetilde\gamma\in SL(3, \C)$
given by:
\[
\widetilde\gamma=
\left (
\begin{array}{lll}
\gamma_{11} & 0           & 0\\
0           & \gamma_{22} & 0\\
0           & 0           & \gamma_{33}\\
\end{array}
\right).
\]
where $\gamma_{11}\gamma_{22}\gamma_{33}=1$. Let  $\ell\in \mathcal{L}\setminus\{X,Y,Z\}$ be a complex line, then either $[e_1]\in \ell$, 
$[e_2]\in \ell$ or $[e_3]\in \ell$. Without loss of generality let us assume that $e_1\in \ell$. Set $\Pi=\Pi_{[e_1],Y}$ and $\pi=\pi_{[e_1],Y}$. Thus $[e_2],[e_3], \pi(\ell\setminus[e_1])\in \bigcap_{\gamma\in \Pi(\Gamma_0)} Fix(\gamma)$ and $[e_2]\neq  \pi(\ell\setminus[e_1])\neq [e_3]$. It falls out that $\Pi(\Gamma_0)=\{ Id \}$. Therefore, for each $\gamma\in \Gamma_0$ there is $\gamma_{11}\in \C^*$ such that $\widetilde \gamma\in SL(3,\C)$ given by:

\begin{equation} \label{e:rara}
\widetilde\gamma=
\left (
\begin{array}{lll}
\gamma_{11}^2 & 0           & 0\\
0           & \gamma_{11}^{-1} & 0\\
0           & 0           & \gamma_{11}^{-1
}\\
\end{array}
\right),
\end{equation}

is a lift of $\gamma$.  Therefore 
\begin{equation}\label{e:fixnor}
\begin{array}{l}
Y\cup \{[e_1]\}\subset \bigcap_{\gamma\in \Pi(\Gamma_0)} Fix(\gamma) \\
Eq(\Gamma_0)= \P^2\setminus (Y\cup \{[e_1]\}).
\end {array}
\end{equation}
\end{proof}

\begin{proposition}
Let $\Gamma\subset PSL(3,\Bbb{C})$ be a discrete subgroup and $\Omega\subset \Bbb{P}^2_{\Bbb{C}}$ be either $Eq(\Gamma)$, $\Omega(\Gamma)$, $\Bbb{P}^2_{\Bbb{C}}\setminus C(\Gamma)$ or a maximal region where $\Gamma$ acts properly  discontinuously. If $LiG(\Omega)=3$ and $Li(\Omega)<\infty$, then $LiG(\Omega)=3$.
\end{proposition}
\begin{proof}
If  the assertion does not follows, then previous lemma yields that $Li(\Omega)=1$, which is a contradiction.
\end{proof}

\begin{proposition}
Let $\Gamma\subset PSL(3,\Bbb{C})$ be a discrete subgroup and $\Omega\subset \Bbb{P}^2_{\Bbb{C}}$ be either $Eq(\Gamma)$, $\Omega(\Gamma)$, $\Bbb{P}^2_{\Bbb{C}}\setminus C(\Gamma)$ or a maximal region where $\Gamma$ acts properly  discontinuously. If $LiG(\Omega)\geq 4$, then  $Li(\Omega)$ is infinite. 
\end{proposition}
\begin{proof}
On the contrary, let us assume that $Li(\Omega)<\infty$. Set  $\mathcal{L}=\{\ell\in Gr(\P^2):\ell\subset \P^2\setminus\Omega\}$ and $\Gamma_0=\bigcap_{\ell\in \mathcal{L}} Isot(\ell,\Gamma)$. Since $LiG(\Omega)\geq 4$, it falls out that each $\gamma\in \Gamma_0$ fixes alt least four lines in general position. Thus $\Gamma_0=\{Id\}$. Since $\Gamma$ is a subgroup of $\Gamma$ with finite index, we conclude that $\Gamma$ is finite. Which is a contradiction.
\end{proof}

Is worth nothing to check that the following theorem follows from the following Lemma. 

\begin{lemma}[See \cite{cano} \& \cite{Nav2}]
Let $\Gamma\subset PSL(3,\Bbb{C})$ be a discrete subgroup and $\Omega\subset \Bbb{P}^2_{\Bbb{C}}$ be either $Eq(\Gamma)$, $\Omega(\Gamma)$, $\Bbb{P}^2_{\Bbb{C}}\setminus C(\Gamma)$ or a maximal region where $\Gamma$ acts properly  discontinuously, then $\Omega$ contains a complex line.
\end{lemma}

\begin{theorem}
Let $\Gamma\subset PSL(3,\Bbb{C})$ be a discrete subgroup and $\Omega\subset \Bbb{P}^2_{\Bbb{C}}$ be either $Eq(\Gamma)$, $\Omega(\Gamma)$, $\Bbb{P}^2_{\Bbb{C}}\setminus C(\Gamma)$ or a maximal region where $\Gamma$ acts properly  discontinuously, then $Li(\Omega)=1,2,3$ or $\infty$.
\end{theorem}

\begin{lemma}
Let $\Gamma\subset PSL(3,\C)$ be a discrete group and $\ell\in Gr(\P^2)$ be  $\Gamma$-invariant  such that the action of $\Gamma$ restricted to $\ell$ is trivial, then:
\begin{enumerate}
 \item \label{1lineq1} If there is an element  in $\Gamma$  with infinite order  with a  diagonalizable lift, then there is  a discrete subgroup $G\subset \C^*$ such that $\Gamma$ is conjugated to:
\[
\Gamma_G= 
\left \{ 
\left (
\begin{array}{lll}
a & 0 & 0\\
0 & a & 0\\
0 & 0 & a^{-2}\\
\end{array}
\right): a\in G
\right\}.
\]
\item \label{1lineq2} If $\Gamma$ does not contains an element with infinite order and with  a diagonalizable lift, then, after conjugating with a projective  transformation if it is necessary,   each element  $\gamma\in \Gamma $ has a lift $\widetilde\gamma\in SL(3,\C)$ given by:
\[
\widetilde\gamma=
\left (
\begin{array}{lll}
a & 0 & b\\
0 & a & c\\
0 & 0 & a^{-2}\\
\end{array}
\right),
\]
where $\mid a\mid=1 $.
\end{enumerate}
\end{lemma}
\begin{proof}
 Let us prove (\ref{1lineq1}).  After conjugating with a projective transformation, if it is necessary, we may assume that $\ell=
\overleftrightarrow{e_1,e_2}$ and there is  $\gamma_0\in \Gamma$ with a lift $\gamma_0\in SL(3,\C)$ given by:
\[
\widetilde\gamma_0=
\left (
\begin{array}{lll}
a & 0 & 0\\
0 & a & 0\\
0 & 0 & a^{-2}\\
\end{array}
\right),
\]
where  $\vert a \vert<1 $. On the other hand, since each element $\gamma\in \Gamma$ has a lift $\widetilde \gamma\in SL(3,\C)$ given by:
\[
\widetilde\gamma=
\left (
\begin{array}{lll}
\gamma_{11} & 0 & \gamma_{13}\\
0 & \gamma_{11} & \gamma_{23}\\
0 & 0 & \gamma_{11}^{-2}\\
\end{array}
\right),
\] 
Set $\Pi_1=\Pi_{e_1,\overleftrightarrow{e_2,e_3}}$, $\pi_1=\Pi_{e_1,\overleftrightarrow{e_2,e_3}}$
and $\Pi_2=\Pi_{e_2,\overleftrightarrow{e_1,e_3}}$, $\pi_2=\Pi_{e_2,\overleftrightarrow{e_1,e_3}}$.  Clearly $Ker(\Pi_j), \,j\in \{1,2\}$, is trivial.\\

Is straightforward to check that the following claim proves the assertion. \\

Claim 1.- For each  $\tau \in \Gamma\setminus\{Id\}$ it follows that
$
Fix(\Pi_j(\tau))=Fix(\Pi_j(\gamma_0)).
$ 
If this is not the case, there are $j_0\in \{1,2\}$ and $\tau_0\in \Gamma$ such that  $Fix(\Pi_{j_0}(\tau_0))\cap Fix(\Pi_{j_0}(\gamma_0))$ is a single point. In consequence $\Pi_{j_0}(\tau_0\gamma_0\tau_0^{-1}\gamma_0^{-1})$ is parabolic. Without loss of generality, let us assume that $j_0=1$, thus $\kappa=\tau_0\gamma_0\tau_0^{-1}\gamma_0^{-1}$ has a lift  $\widetilde\kappa\in SL(3,\C)$ given by:
\[
\widetilde\kappa=
\left (
\begin{array}{lll}
1 & 0 & \kappa_{13}\\
0 & 1 & \kappa_{23}\\
0 & 0 & 1\\
\end{array}
\right),
\] 
where $\kappa_{13}\neq 0$. Finally, a simple calculation shows:
\[
\widetilde\gamma_0^m\widetilde\kappa 
\widetilde\gamma_0^{-m}=
\left (
\begin{array}{lll}
1 & 0 & a^{-m}\kappa_{13}\\
0 & 1 & a^{-m}\kappa_{23}\\
0 & 0 & 1\\
\end{array}
\right) \xymatrix{
\ar[r]_{m \rightarrow  \infty}&} 
 \left (
\begin{array}{lll}
1 & 0 & 0\\
0 & 1 & 0\\
0 & 0 & 1\\
\end{array}
\right),
\]
which is a contradiction since $\Gamma$ is discrete.\\

Let us prove (\ref{1lineq2}). After conjugating with a projective transformation, if it is necessary, we may assume that $\ell=\overleftrightarrow{e_1,e_2}$. In consequence each element $\gamma\in \Gamma$ has a lift given by:
\[
\widetilde\gamma=
\left (
\begin{array}{lll}
\gamma_{11} & 0 & \gamma_{13}\\
0 & \gamma_{11} & \gamma_{23}\\
0 & 0 & \gamma_{11}^{-2}\\
\end{array}
\right).
\]
Set $\Pi_1=\Pi_{e_1,\overleftrightarrow{e_2,e_3}}$, $\pi_1=\Pi_{e_1,\overleftrightarrow{e_2,e_3}}$, $\Pi_2=\Pi_{e_2,\overleftrightarrow{e_1,e_3}}$, $\pi_2=\Pi_{e_2,\overleftrightarrow{e_1,e_3}}$ and $\lambda_\infty:(\Gamma)\rightarrow \C^* $ given by $\lambda_\infty(\gamma)=\gamma_{11}^3$.\\

Clearly, the assertion yields from the following claim:\\

Claim 2.- $\lambda_\infty(\Gamma)\subset \Bbb{S}^1 $. On the contrary, let us assume that there is $\gamma_0\in \Gamma$ such that $\vert \lambda_\infty(\gamma_0)\vert \neq 1$. In such case:
\[
Fix(\Pi_1(\gamma_0))=\{p_
1=[0:\gamma_{23}^{(0)}:1-\gamma_{11}^{(0)}]
,[e_2]\}
\]
\[
Fix(\Pi_2(\gamma_0))=\{p_2=[\gamma_{13}^{(0)}:0:1-\gamma_{11}^{(0)}]
,[e_1]\}.
\]
Set $\ell_1=\overleftrightarrow{p_1,e_1}$ and $\ell_2=\overleftrightarrow{p_2,e_2}$. Thus $\gamma_0(\ell_1)=\ell_1$ and $\gamma_0(\ell_2)=\ell_2$. That is the unique point $*$ in  $\ell_1\cap\ell_2$ is fixed by $\gamma_0$ and $*\in \P^2\setminus(\ell_1\cup\ell_2$. Therefore $\gamma_0$ is diagonalizable. Finally, since $\lambda_\infty$ is a group morphism, we conclude that 
$\gamma_0$ has infinite order. Which contradicts the initial assumption.
\end{proof}

Now it follows easily that:

\begin{corollary} \label{equiline}
 Let $\Gamma\subset PSL(3,\C)$ be a discrete group and $\ell\in Gr(\P^2)$ be  an invariant line under $\Gamma$ such that the action of $\Gamma$ restricted to $\ell$ is trivial, then there is a point $p\in \P^2$ such that  $Eq(\Gamma)=\Bbb{P}^2\setminus (\ell\cup \{p \})$.
\end{corollary}

\begin{lemma}
Let $\Gamma\subset PSL(3,\C)$ be a discrete group such that each element $\gamma\in \Gamma$ has a lift $\tilde \gamma\in GL(3,\C)$ given by:
\[
\tilde{
\gamma}=
\left (
\begin{array}{lll}
\gamma_{11} & 0          & 0\\
0           &\gamma_{22} & \gamma_{23}\\
0           &\gamma_{32} & \gamma_{33}\\
\end{array}
\right ),
\]
then $LiG(Eq(\Gamma))\leq 3.$
\end{lemma}
\begin{proof}
Let $(\gamma_m)\subset \Gamma$ be a sequence of distinct, thus for each $m\in \Bbb{N}$ there is $\tilde{\gamma}_m\in GL(3,\Bbb{C})$ a lift of $\gamma_m$ which is given by:
\[
\tilde{
\gamma}_m=
\left (
\begin{array}{lll}
\gamma_{11}^{(m)} & 0          & 0\\
0           &\gamma_{22}^{(m)} & \gamma_{23}^{(m)}\\
0           &\gamma_{32}^{(m)} & \gamma_{33}^{(m)}\\
\end{array}
\right )
\]
where $max\{\mid \gamma_{11}^{(m)}\mid, \mid \gamma_{22}^{(m)}\mid, \mid \gamma_{33}^{(m)}\mid, \mid \gamma_{23}^{(m)}\mid, \mid \gamma_{32}^{(m)}\mid \}=1$. Since the sequences $(\gamma_{11}^{(m)})$,$(  \gamma_{22}^{(m)})$, $( \gamma_{33}^{(m)})$, $( \gamma_{23}^{(m)})$, $(\gamma_{32}^{(m)}) $ are bounded, there is a sequence $(k_m)\subset (m)$ and $\gamma_{11}$, $ \gamma_{22}$, $ \gamma_{33}$, $ \gamma_{23}$, $\gamma_{32} \in \C$   such that 
\[
\tilde{\gamma}_{k_m} \xymatrix{
\ar[r]_{m \rightarrow  \infty}&}
\left (
\begin{array}{lll}
\gamma_{11} & 0          & 0\\
0           &\gamma_{22} & \gamma_{23}\\
0           &\gamma_{32} & \gamma_{33}\\
\end{array}
\right )=\tilde \gamma
\]
where $\gamma_{11}( \gamma_{22}  \gamma_{33}-\gamma_{32}\gamma_{23})=0$. In consequence  $Ker([[\tilde \gamma]])$ is either $[e_1]$, a line passing trough $[e_1]$, $\overleftrightarrow{[e_2],[e_3]}$ or a point in $\overleftrightarrow{[e_2],[e_3]}$. Now, the conclusion follows easily. 
\end{proof}

\begin{proposition}
Let $\Gamma\subset PSL(3,\C)$ be a discrete group and  $\Omega\subset \P^2$ be a  discontinuity region for the action of $\Gamma$ which is either $Eq(\Gamma)$, $\Omega(\Gamma)$ or a maximal discontinuity region. If $4\leq LiG(\Omega)<\infty$, then for each $\mathcal{L}\in GL(\Omega)$ with $Card(\mathcal{L}=LiG(\Omega))$ it falls out that:
\begin{enumerate}
\item \label{4lin1} There is at least one   vertex of  $\mathcal{L}$;
\item \label{4lin2} If $v\in \P^2$ is a vertex of $\mathcal{L}$, then $Isot(v,\Gamma)$ is a subgroup of $\Gamma$ with finite index;
\item \label{4lin3} There are at least two vertexes  of  $\mathcal{L}$;
\item \label{4lin4} The vertexes of $ \mathcal{L}$ lie in a complex line;
\item \label{4lin5}  There are at most two vertexes  of  $\mathcal{L}$;
\item \label{4lin6} If $\ell\in Gr(\P^2)$ does not contains a vertex of $\mathcal{L}$, then its orbit under $\Gamma$ is infinite.
\end{enumerate}
 
\end{proposition}
\begin{proof}
Let us prove part (\ref{4lin1}).
 Since $LiG(\Omega)\geq 4$ we conclude that $Lin(\Omega)=\infty$. On the other hand, since $LinG(\Omega)<\infty$, we deduce that there is $p_0\in \Bbb{P}^2_{\Bbb{C}}$, $\ell_1,\ell_2\in \mathcal{L}$ and a infinite set  $\mathcal{C}_0$ such that 
$ \{p_0\}= \bigcap\mathcal{C}_0=\ell_1\cap \ell_2$ and $\bigcup\mathcal{C}_0\subset \P^2\setminus\Omega$.\\

Let us prove (\ref{4lin2}).  Let $p_0\in \widetilde{\mathcal{P}}(\mathcal{L})$, then there are  $\ell_1,\ell_2\in \mathcal{L}$ and a infinite set  $\mathcal{C}_0$ such that 
$ \{p_0\}= \bigcap\mathcal{C}_0=\ell_1\cap \ell_2$ and $\bigcup\mathcal{C}_0\subset \P^2\setminus\Omega$. Let $\gamma\in \Gamma$, then there is $p_1\in \P^2$, $\widetilde{\ell_1},\widetilde{\ell_2}\in \mathcal{L}$ and a infinite set  $\mathcal{C}_1\subset \mathcal{C}_0$ such that 
$ \{p_1\}= \bigcap\gamma(\mathcal{C}_0)=\widetilde{\ell_1}\cap\widetilde{ \ell_2}$. In consequence $\gamma(p_0)=p_1$. Which proves the assertion.\\

Let us prove (\ref{4lin3}). Let $p_0\in \P^2$ be a vertex of $\mathcal{L}$, thus there are $\ell_3,\ell_3\in \mathcal{L}$ such that $p_0\notin(\ell_3,\ell_4)$. Then the orbits of $\ell_3$ or $\ell_4$ under $\Gamma$ are  infinite. Observe that, if there is $\gamma\in \Gamma$ such that $\gamma(p_0)\in \ell_1$, the assertion follows easily. So let us assume that  $p_0\notin \Gamma \ell_1$. Since $LiG(\Omega)=k_0<\infty$ and $\Gamma \ell_1$ is infinite, it falls out that there is  $p_1\in \P^2\setminus\{p_0\}$, $\ell_1,\ell_2\in \mathcal{L}$ and a infinite set  $\mathcal{C}_3\in Gr(\P^2)$ an infinite set  such that 
$ \{p_1\}= \bigcap \mathcal{C}_3=\widetilde{\ell_1}\cap\widetilde{ \ell_2}$ and $\bigcup \mathcal{C}_3\subset \P^2\setminus \Omega$. Which concludes the proof.\\

Let us prove (\ref{4lin4}). On the contrary, let us assume that the vertex of $\mathcal{L}$ does not lie in a complex line, thus there is a set $C\subset \P^2$ of vertexes of $\mathcal{L}$ such that $Card(C)=3$. Let   $\Gamma_0=\bigcup_{x\in C} Isot(x, \Gamma) $, thus $\Gamma_0$ is a  subgroup of $\Gamma$ with finite index. Thus after conjugating by a projective transformation, if it is necessary, we can assume that $C=\{[e_1],[e_2],[e_3]\}$.  In consequence, each element $\gamma\in \Gamma$ has a lift $\widetilde \gamma\in SL(3,\C)$ given by:
\[
\widetilde\gamma=
\left (
\begin{array}{lll}
\gamma_{11} & 0           & 0\\
0           & \gamma_{22} & 0\\
0           & 0           & \gamma_{33}\\
\end{array}
\right).
\]
where $\gamma_{11}\gamma_{22}\gamma_{33}=1$. Which is a contradiction.  \\

Let us prove (\ref{4lin5}). If this is not the case, then there is a set $C_1$ of vertexes of $\mathcal{L}$ such that $Card(C_1)=3$. Set  $\Gamma_0=\bigcap_{v\in C_1} Isot(v,\Gamma) \subset $ is a subgroup of $\Gamma$ with finite index fixing 
at least 3 points in the line $\ell$ which contains $C_1$. Thus $\Gamma_0$ fixs each point in $\ell$. By corollary (\ref{equiline}), it follows that $Eq(\Gamma)=\P^2\setminus \ell=\P^2\setminus L_0(\Gamma_0)$, which is a contradiction. \\

Let us prove (\ref{4lin6}). On the contrary, let us assume that there is $\ell_0\in Gr(\P^2)$ an element with infinite orbit which is  disjoint from the vertexes of $\mathcal{L}$. Thus   
$\Gamma_0=Isot(\ell_0,\Gamma) \cap(\bigcap_{v\in V}	 Isot(v,\Gamma))$,   is a subgroup of $\Gamma$ with finite index fixing $\ell_0$ and each vertex $\mathcal{L}$, here $V$ is the set of vertex of $\mathcal{L}$. Let $\ell_V$ be the line which contains the vertexes of $\mathcal{L}$. Then  $\Gamma_0$ leaves invariant $\ell_0\cap \ell$. In consequence leaves invariant 3 points in $\ell$. Therefore, $\Gamma_0$ leaves invariant each point in $\ell$. By corollary (\ref{equiline}), it follows that there is a point $q\in \P^2$ such that $Eq(\Gamma)=\P^2\setminus (\ell\cup\{q\})=\P^2\setminus L_0(\Gamma)$. Which is a contradiction. 
\end{proof}

\begin{theorem}
Let $\Gamma\subset PSL(3,\Bbb{C})$ be a discrete subgroup and $\Omega\subset \Bbb{P}^2_{\Bbb{C}}$ be either $Eq(\Gamma)$, $\Omega(\Gamma)$, $\Bbb{P}^2_{\Bbb{C}}\setminus C(\Gamma)$ or a maximal region where $\Gamma$ acts properly  discontinuously, then $LiG(\Omega)=1,2,3,4$ or $\infty$.
\end{theorem}
\begin{proof}
On the contrary let us assume that $4<LiG(\Omega)<\infty$, then there is $\mathcal{L}\in LG(\Omega)$ such that $Card(\mathcal{L})=LiG(\Omega)$. Let $p_0,p_1$ be the vertex of $\mathcal{L}$. Then there is $\ell\in \mathcal{L}$, disjoint from $\{p_0,p_1\}$. In consequence, previous Lemma yields that $\Gamma \ell$ is an infinite set in $Gr(\P^2)$. Thus there is $p\in \{p_0,p_1\}$ and $(\gamma_m)\subset \Gamma$ such that $\gamma_m \ell$ is an infinite set in  $Gr(\P^2)$ and $\bigcap_{m\in \Bbb{N}} \gamma_m\ell=\{p\}$. Then $\{\gamma_{m}^{-1}(p):m\in \Bbb{N}\}\subset \ell\cap \{p_0,p_1\}$. Which is a contradiction.
\end{proof}

\section{The Limit Set}\label{s:limit}

\begin{lemma}
Let $\Gamma\subset PSL(3,\Bbb{C})$ be a discrete group which satisfies  $LiG(\Bbb{P}^2_{\Bbb{C}}\setminus C(\Gamma))=1$. If $\ell\in Gr(\P^2)$ is the unique line in $\P^2\setminus C(\Gamma)$ and $\tilde \Gamma$ denotes the action of $\Gamma$ restricted to $\ell$, then $\tilde \Gamma$ does not contains loxodromic elements.
\end{lemma}
\begin{proof} On the contrary let us assume that there is
 $\gamma_0\in \Gamma$ such that $\gamma_0\mid_{\ell}$ is loxodromic. Thus Normal Jordan Form Theorem yields that there is $\tilde{\gamma}_{0}\in GL(3,\C)$ a lift of $\gamma_0$ and $\gamma\in GL(3,\C)$ that $[[\gamma]](\ell)=\overleftrightarrow{[e_1],[e_2]}$ and $\gamma\tilde{\gamma}_{0}\gamma^{-1}$ is either:
 \[
A=\left (
\begin{array}{lll}
\lambda_1  & 0            & 0\\
0            &  \lambda_2 & 0\\
0            &  0            & 1\\
\end{array}
\right);\,
C=
\left (
\begin{array}{lll}
\lambda    & 0 & 0\\
0            & 1 & 1\\
0            & 0 & 1\\
\end{array}
\right), 
\]
where $\vert \lambda_1\vert<\vert \lambda_2\vert$ and $\vert \lambda\vert\neq 1$. A simple calculation shows:
\[
\P^2\setminus Eq(<[A]>)=
\left \{
\begin{array}{ll}
\overleftrightarrow{[e_2],[e_3]}\cup \{[e_1]\} & \textrm{ if } \vert \lambda_2 \vert=1;\\
\overleftrightarrow{[e_1],[e_3]}\cup \{[e_2]\} & \textrm{ if } \vert \lambda_1 \vert=1;\\
\overleftrightarrow{[e_2],[e_1]}\cup \overleftrightarrow{[e_2],[e_3]} & \textrm{ if } \vert \lambda_2\vert <1;\\
\overleftrightarrow{[e_1],[e_3]}\cup \overleftrightarrow{[e_2],[e_3]} & \textrm{ if } \vert \lambda_1\vert < 1 < \vert \lambda_2 \vert;\\
\overleftrightarrow{[e_1],[e_3]}\cup \overleftrightarrow{[e_2],[e_1]} & \textrm{ if } \vert \lambda_1\vert>1;\\
\end{array}
\right.
\]
\[
\P^2\setminus Eq(<[B]>)=
\overleftrightarrow{[e_2],[e_1]}\cup \overleftrightarrow{[e_2],[e_3]} .
\]
Which is not possible, since $[[\gamma]](\ell)=\overleftrightarrow{[e_1],[e_2]}$.
\end{proof}
\begin{lemma} [See \cite{cano}]
Let $\Gamma\subset PSL(2,\C)$  be a group without loxodromic elements, then $\Gamma$ is conjugated to a subgroup of either $SO(3)$ or $EP(\C)$, where
\[
EP(\C)=
\left \{
\left (
\begin{array}{ll}
a & b\\
0 & 1
\end{array}
\right )
: 
\vert a\vert=1  \textrm{ and } b\in \C
\right \}.
\] 
\end{lemma}
\begin{lemma} \label{l:lic1}
Let $\Gamma\subset PSL(3,\Bbb{C})$ be a discrete group such that $LiG(\Bbb{P}^2_{\Bbb{C}}\setminus C(\Gamma))=1$. If $\ell$ is the unique line in $C(\Gamma)$, it holds:
\begin{enumerate}
\item \label{i:1lic1} If $\Gamma\vert_\ell $ is conjugated to a group of $S0(3)$, then $LiG(Eq(\Gamma))=1$;
\item \label{i:2lic1} If $\Gamma\vert_\ell $ is conjugated to a group of $EP(\C)$, then either $LiG(Eq(\Gamma))\leq 2$ or $Eq(\Gamma)=\emptyset$.
\end{enumerate}
\end{lemma}
\begin{proof}
Let us prove (\ref{i:1lic1}).  After conjugating with a projective transformations if it is necessary, we may assume that $\ell=\overleftrightarrow{[e_2],[e_3]}$ and $\Gamma\vert_{\ell}$ is a subgroup of $SO(3)$. Now, 
let $(\gamma_m) \subset \Gamma$ be a sequence of distinct elements, then for each $m\in \Bbb{N}$ there is $a_m\in \C $, $b_m=(b_{j1}^{(m)})\in M(2\times 1,\C)$, $\tilde{\gamma}_m\in GL(3,\C)$ and $H_m=(h_{ij}^{(m)})\in SO(3)$ such that $\tilde{\gamma_0}$ is a lift of $\gamma_0$ and 
\[
\tilde{\gamma}_m=
\left (
\begin{array}{ll}
a_m &0\\
b_m & H_m\\
\end{array}
\right ).
\]
On the other hand, since $SO(3)$ and $\P^1$ are compact, we may assume that there is $a\in \P^1$ and $H\in SO(3)$ such that $a_m \xymatrix{
\ar[r]_{m \rightarrow  \infty}&}a$ and $H_m\xymatrix{
\ar[r]_{m \rightarrow  \infty}&} H$. Define $\vert \gamma_m\vert=max\{\vert a_m\vert, \vert h_{ij}^{(m)} \vert, \vert b_{j1}^{(m)}\vert, \,j,i\in \{1,2\}  \} $, we must consider the following cases:\\

Case 1. $(\vert \gamma_m\vert)$ is unbounded. Thus there is a subsequence of $(\vert \gamma_m\vert)$, which is still denoted $(\vert \gamma_m\vert)$, such that $\vert \gamma_m\vert \xymatrix{
\ar[r]_{m \rightarrow  \infty}&}  \infty$. In consequence $(\vert \gamma_m\vert^{-1}\tilde{\gamma}_m)$ is a bounded sequence in $M(3,\C)$, thus there is $\gamma\in M(3,\C)\setminus \{0\}$ and a subsequence of $(\vert \gamma_m\vert^{-1}\tilde{\gamma}_m)$, still denoted $(\vert \gamma_m\vert^{-1}\tilde{\gamma}_m)$, which satisfies
$\vert \gamma_m\vert^{-1} \tilde {\gamma}_m\xymatrix{
\ar[r]_{m \rightarrow  \infty}&}  \gamma$. Since $\vert \gamma_m\vert^{-1} H_m\xymatrix{
\ar[r]_{m \rightarrow  \infty}&}  0$, we conclude that there is $a\in \C$ and $b\in M(2\times 1,\C)$ such that:
\[
\gamma=
\left (
\begin{array}{ll}
a &0\\
b &0\\
\end{array}
\right ).
\]
In consequence $Ker(\gamma)=\overleftrightarrow{[e_2],[e_3]}$.\\

Case 2. $(\vert \gamma_m\vert)$ is bounded. Since $\Gamma$ is discrete, we conclude that   $a_m \xymatrix{
\ar[r]_{m \rightarrow  \infty}&}  0$. Moreover, $(\tilde{\gamma}_m)$ is a bounded sequence in $M(3,\C)$, thus there is $\gamma\in M(3,\C)\setminus \{0\}$ and a subsequence of $( \tilde{\gamma}_m)$, still denoted $(\tilde{\gamma}_m)$, which satisfies
$ \tilde {\gamma}_m\xymatrix{
\ar[r]_{m \rightarrow  \infty}&}  \gamma$. Thus there are $b\in M(2\times 1,\C)$  such that: 
\[
\gamma=
\left (
\begin{array}{ll}
0 & 0 \\
b & H
\end{array}
\right ).
\]
A straighfoward calculation shows that $w_m=(1,-(H_m-a_m)^{-1}(b_m))$ is proper vector with proper value $a_m$. Moreover, it holds 
\[
w_m \xymatrix{
\ar[r]_{m \rightarrow  \infty}&} w=(1,-H^{-1}(b))
\]
and $\gamma(w)=0$. That is $Ker([[\gamma]])\subset  L_0(\Gamma)\subset C(\Gamma)$.\\

From the previous cases, it follows that  $LiG(\Gamma)=1$.\\

Let us prove (\ref{i:2lic1}).  After conjugating with a projective transformations if it is necessary, we may assume that $\ell=\overleftrightarrow{[e_2],[e_3]}$ and $\Gamma\vert_{\ell}$ is a subgroup of $EP(\C)$. Now, 
let $(\gamma_m) \subset \Gamma$ be a sequence of distinct elements, then for each $m\in \Bbb{N}$ there are $a_m,b_m,c_m,d_m\in \C $, $\theta_m\in \Bbb{R}$ and $\tilde{\gamma}_m\in GL(3,\C)$  such that $\tilde{\gamma_0}$ is a lift of $\gamma_0$ and 
\[
\tilde{\gamma}_m=
\left (
\begin{array}{lll}
a_m & 0 & 0\\
b_m & e^{\pi i\theta_m} &d_m\\
c_m & 0 & 1\\
\end{array}
\right ).
\]
 Define $\vert \gamma_m\vert=max\{\vert a_m\vert,\vert b_m\vert, \vert c_m\vert, \vert d_m\vert \} $, we must consider the following cases:\\

Case 1. $(\vert \gamma_m\vert)$ is unbounded. Thus there is a subsequence of $(\vert \gamma_m\vert)$, which is still denoted $(\vert \gamma_m\vert)$, such that $\vert \gamma_m\vert \xymatrix{
\ar[r]_{m \rightarrow  \infty}&}  \infty$. In consequence $(\vert \gamma_m\vert^{-1}\tilde{\gamma}_m)$ is a bounded sequence in $M(3,\C)$, thus there is $\gamma\in M(3,\C)\setminus \{0\}$ and a subsequence of $(\vert \gamma_m\vert^{-1}\tilde{\gamma}_m)$, still denoted $(\vert \gamma_m\vert^{-1}\tilde{\gamma}_m)$, which satisfies
$\vert \gamma_m\vert^{-1} \tilde {\gamma}_m\xymatrix{
\ar[r]_{m \rightarrow  \infty}&}  \gamma$. Thus there are $a,b,c,d\in \C$ such that: 
\[
\gamma=
\left (
\begin{array}{lll}
a & 0 & 0\\
b & 0 &d\\
c & 0 & 0\\
\end{array}
\right ).
\]
A straightforward calculation shows:
\[
Ker(\gamma)=
\left \{
\begin{array}{ll}
\overleftrightarrow{[e_2],[e_1]} & \textrm{ if } d=0;\\
\textrm{a line passing trough } [e_2] & \textrm{ if } d\neq 0,\, $a=b=0$;\\
\{e_1\} &\textrm{ in any other case }\\
\end{array}
\right .
\]

Case 2. $(\vert \gamma_m\vert)$ is bounded. Since $\Gamma$ is discrete, we conclude that   $a_m \xymatrix{
\ar[r]_{m \rightarrow  \infty}&}  0$. Moreover, $(\tilde{\gamma}_m)$ is a bounded sequence in $M(3,\C)$, thus there is $\gamma\in M(3,\C)\setminus \{0\}$ and a subsequence of $( \tilde{\gamma}_m)$, still denoted $(\tilde{\gamma}_m)$, which satisfies
$ \tilde {\gamma}_m\xymatrix{
\ar[r]_{m \rightarrow  \infty}&}  \gamma$. Thus there are $b,c,d\in \C$ and $\theta\in \Bbb{R}$ such that: 
\[
\gamma=
\left (
\begin{array}{lll}
0 & 0 & 0\\
b & e^{\pi i\theta} &d\\
c & 0 & 1\\
\end{array}
\right ).
\]
Set 
$
K_m=
\left(
\begin{array}{ll}
e^{\pi i \theta_m}  & d_m \\ 
0 & 1
\end{array}
\right )$ and 
$	
K=
\left ( 
\begin{array}{ll}
e^{\pi i \theta}  & d \\ 
0 & 1
\end{array}
\right )$. A straighfoward calculation shows that $w_m=(1,-(K_m-a_m)^{-1}(b_m,c_m))$ is proper vector with proper value $a_m$. Moreover, it holds 
\[
w_m \xymatrix{
\ar[r]_{m \rightarrow  \infty}&} w=(1,-K^{-1}(b,c))
\]
and $\gamma(w)=0$. That is $Ker([[\gamma]])\subset  L_0(\Gamma)\subset C(\Gamma)$.\\

The previous cases, yield  that either $Eq(\Gamma)=\emptyset$ or $LiG(\Gamma)\leq2$.
\end{proof}

\begin{lemma} \label{l:tec2}
Let $\Gamma\subset PSL(3,\C)$ be a discrete group,  such that $LiG(\P^2\setminus C(\Gamma))=2$, then:
\begin{enumerate}
\item \label{l:1tec2} There is $\tau\in PSL(3,\C)$ a projective transformation such that  each $\gamma\in \tau\Gamma\tau^{-1}$ has a lift $\tilde{\gamma}\in SL(3,\C)$ such that $\tilde{\gamma}$ is given by:
\[
\gamma=
\left (
\begin{array}{lll}
\gamma_{11} & \gamma_{12} & \gamma_{13}\\
0           & \gamma_{22} & \gamma_{23}\\
0           & \gamma_{32} & \gamma_{33}\\
\end{array}
\right );
\]
 \item \label{l:3tec2} Let $\gamma\in \tau\Gamma\tau^{-1}$ and  $\tilde{\gamma}\in SL(3,\C)$ a lift of $\gamma$, then there is a proper values $\lambda$ of $\left ( 
\begin{array}{ll}
\gamma_{22}  & \gamma_{23} \\ 
\gamma_{32} & \gamma_{33}
\end{array}
\right )$ such that $$\vert  \gamma_{11} \vert \leq \vert \lambda \vert;$$
\item \label{l:4tec2} It holds either $LiG(Eq(\Gamma))$ or $Eq(\Gamma)=\emptyset$.
\end{enumerate}
\end{lemma}
\begin{proof}
Let us prove (\ref{l:1tec2}). It follows easily form lemma \ref{l:pfix}.\\

Let us prove (\ref{l:3tec2}). Let $\gamma\in \tilde\Gamma$ and $\tilde\gamma\in SL(3,\C)$ be a lift of $\gamma$, then there is $H\in GL(2,\C)$, $a\in \C^*$ and $b\in M(1\times 2,\C)$ such that 
\[
\tilde\gamma=
\left (
\begin{array}{ll}
a & b \\
0 & H\\
\end{array}
\right )
\]
Let us assume that $\vert a\vert > max\{\vert \lambda\vert:\lambda \textrm{ is a proper value of } H \}$. Since every proper value of $ H$ is a proper value of $\tilde \gamma$, it follows that  $\tilde \gamma$ has at least two different proper values. Thus normal Jordan form theorem yields that    $\tilde \gamma$ is either conjugated to
\[
A=
\left (
\begin{array}{lll}
a & 0 &0 \\
0 & \lambda_{1}& 0 \\
0 &           0& \lambda_{2} \\
\end{array}
\right );\, 
B=
\left (
\begin{array}{lll}
a & 0 &0 \\
0 & \lambda_{1}& 1 \\
0 &           0& \lambda_{1} \\
\end{array}
\right )
\]
clearly, in the respective cases, $\lambda_{1},\lambda_2$ are proper values of $H$. 
A simple calculation shows:
\[
\Lambda(<[A]>)=
\left \{
\begin{array}{ll}
\overleftrightarrow{[e_1],[e_3]}\cup 
\overleftrightarrow{[e_3],[e_2]} &
\vert \lambda_1\vert <\vert \lambda_2 \vert;\\ 
\overleftrightarrow{[e_1],[e_2]}\cup 
\overleftrightarrow{[e_3],[e_2]} &
\vert \lambda_2\vert <\vert \lambda_1 \vert;\\
\{[e_1]\}\cup 
\overleftrightarrow{[e_3],[e_2]} &
\vert \lambda_1\vert =\vert \lambda_2 \vert\\
\end{array}
\right.
\]  
and
\[
\Lambda(<[B]>)=
\overleftrightarrow{[e_1],[e_2]}\cup 
\overleftrightarrow{[e_3],[e_2]}
\]
in any case, it follows that $LiG(\Gamma)\geq 3$. Which is a contradiction.\\

Let us prove (\ref{l:4tec2}). Let $(\gamma_m)\subset \Gamma$ be a sequence of distinct elements and $\gamma\in SP(3,\C)$ such that $\gamma_m \xymatrix{
\ar[r]_{m \rightarrow  \infty}&} 
\gamma$. Thus for each $m\in n\Bbb{N}$ there is $\tilde{\gamma}_m\in SL(3,\C)$, $\gamma_{11}^{(m)}\in \C^*$, $\gamma_{12}^{(m)}\in M(1\times  2,\C)$ and $H_m\in GL(2,\C) $ such that $\tilde{\gamma}_m$ is a lift of $\gamma_m$ and
\[
\tilde{\gamma}_m=
\left (
\begin{array}{ll}
\gamma_{11}^{(m)} & \gamma_{12}^{(m)} \\
0 & H_m\\
\end{array}
\right ).
\]
Set $\vert \vert H_m\vert \vert=max\{H_m(x):\vert x\vert =1\}$. Thus $(\vert \vert H_m\vert \vert^{-1}H_m)$ is sequence of unitary operators, thus we can assume that there is $H\in GL(2,\C)$ a unitary  operator such that $\vert \vert H_m\vert \vert^{-1}H_m \xymatrix{
\ar[r]_{m \rightarrow  \infty}&} 
H $, point wise. On the other hand, part  $(\ref{l:3tec2})$ of the present lemma
yields $\vert \gamma_{11}^{(m)}\vert\leq \vert \vert H_m\vert \vert$. Then we can assume that there is $\gamma_{11}\in \C$ such that $ \gamma_{11}^{(m)} \vert \vert H_m\vert \vert^{-1} \xymatrix{
\ar[r]_{m \rightarrow  \infty}&} \gamma_{11}$. Set $\gamma_{12}^{(m)}=(a_m,b_m)$ and $\vert \gamma_{12}{(m)}\vert=max\{\vert a_m\vert , \vert b_m\vert \} $. And consider the following cases:\\

Case 1.- The sequence $(\vert \vert H_m\vert \vert^{-1}\vert \gamma_{12}^{(m)}\vert) $ is unbounded. In this case, we may assume that $\vert \vert H_m\vert \vert^{-1}\vert \gamma_{12}^{(m)}\vert  \xymatrix{
\ar[r]_{m \rightarrow  \infty}&} \infty$. Thus a simple calculation shows:
\[
\vert \gamma_{12}^{(m)}\vert^{-1}\vert \vert H_m\vert \vert\tilde \gamma_{m} \xymatrix{
\ar[r]_{m \rightarrow  \infty}&} 
\left (
\begin{array}{ll}
0 & \gamma_{12}\\
0 & 0\\
\end{array}
\right ).
\]
which shows that $Ker(\gamma)$ is a complex line passing trough $[e_1]$. \\

Case 2.- The sequence $(\vert \vert H_m\vert \vert^{-1}\vert \gamma_{12}^{(m)}\vert) $ is bounded. In this case, we may assume that there is $\gamma_{12}\in M(1\times  2,\C)$ such that $\vert \vert H_m\vert \vert^{-1} \gamma_{12}^{(m)}  \xymatrix{
\ar[r]_{m \rightarrow  \infty}&} \gamma_{12}$. A simple calculation shows:
\[
\vert \vert H_m\vert \vert^{-1}\tilde \gamma_{m} \xymatrix{
\ar[r]_{m \rightarrow  \infty}&} 
\left (
\begin{array}{ll}
\gamma_{11} & \gamma_{12}\\
0 & H\\
\end{array}
\right ).
\]
Now is straighfoward to check that  $Ker(\gamma)=[e_1]$ is either $[e_1]$, a line passing trough $[e_1]$ or a point in $\P^2$. Let us assume that  $Ker(\gamma)$ is a point $*$. In thus case there is $\tilde\ell\in Gr(\P^2)$ such that $*\in \ell\subset C(\Gamma)$. 
It can be shown that, see \cite{bcn},  there is a  complex line $\hat \ell\in Gr(\P^2)$ and subsequence of $(\gamma_m)$, still denoted $(\gamma_m)$,  such that 
\begin{equation}\label{e:cg}
\gamma ^{-1}_m \xymatrix{
\ar[r]_{m \rightarrow  \infty}&}  * \textrm{ uniformly  on compact sets of }\P^2\setminus \hat\ell.
\end{equation}
Let $\check \ell\in Gr(\P^2)$ be such that $\check\ell\subset 	C(\Gamma) $ and $\check\ell\neq \hat\ell$. Thus equation \ref{e:cg}, yields $\gamma_m^{-1}(\check \ell) \xymatrix{
\ar[r]_{m \rightarrow  \infty}&} \overleftrightarrow{*,[e_1]} $. Which clearly concludes the proof.
\end{proof}
From the previous lemmas  is straighfoward to check 

\begin{corollary}
Let $\Gamma\subset PSL(3,\Bbb{C})$ be a  complex Kleinian group such that $LiG(\Omega(\Gamma))\geq 3$, then $LiG(\Bbb{P}^2_{\Bbb{C}}\setminus C(\Gamma))>3$.
\end{corollary}

Now Theorem \ref{t:main3} follows easily.

\section{Examples} \label{s:ex}

\subsection{A group With infinite  Lines in General Position}
Let $\Gamma\subset PU(1,2)$ be a discrete group which acts with compact quotient on $\Bbb{H}^2_{\Bbb{C}}$, see \cite{DM}, then it can be proved that, see \cite{Nav1}, $\Omega(\Gamma)=\Bbb{H}^2_{\Bbb{C}}$ and $LiG(\Omega)=\infty$.

\subsection{A group With 4 Lines in General Position}
Let $M$ be the matrix $M=\left (
\begin{array}{ll}
3 & 5\\
-5 & 8
\end{array}
\right ) $.  A straightforward calculation shows, that $M$  has eigenvalues
\[
\alpha_{\pm}=\frac{-5\pm\sqrt{21}}{2},
\] 
with respective eigenvectors:
\[
v_{+}=\left (1, \frac{-11+\sqrt{21}}{10}\right );\,
 v_{-}=\left (
\frac{-11+\sqrt{21}}{10},1\right ).
\] 
Let
$\Gamma^\ltimes_M\subset PSL(3,\Bbb{C})$ be the group 
generated by:
\[
\begin{array}{l}
\gamma_0(w,z)=(\alpha_+ w,\alpha_- z);\\
\gamma_1(w,z)=(w+1,z+\frac{-11+\sqrt{21}}{10});\\
\gamma_2(w,z)=(w+\frac{-11+\sqrt{21}}{10},z+1);\\
\gamma_3(w,z)=(z,w).
\end{array}
\]
Then it is easily seen that   $\Omega(\Gamma_A^{\ltimes})=\bigcup_{j,j=0,1}(\H^{(-1)^i}\times \Bbb{H}^{(-1)^j})$ and $LiG(\Omega(\Gamma_A^{\ltimes}))=4$. Moreover, one can check that this is "essentially" the only way to construct examples of groups $\Gamma\in PSL(3,\Bbb{C})$ with $LiG(\Omega(\Gamma))=4$, see \cite{bcn2}.

\subsection{A group With Three Lines in General Position and infinite Lines}
Let $\Gamma\subset PSL(2,C)$ be a non-elementary kleinian group and
$G\subset \mathbb{C}^*$ be  a discrete subgroup and   for each $g\in \Bbb{C}^*$ set  $i_g:SL(2, \mathbb{C})\rightarrow SL(3,  \mathbb{C})$
   given by:
\[
i(h)= \left (
\begin{array}{ll}
gh & 0\\
0 & g^{-2}
\end{array}
\right ).
\]
Then $S(\Gamma,G)$ is to be defined as:
\[
S(\Gamma,G)= \{i_g( \gamma):g\in G\, \&\,[[\gamma]]\in  \Gamma\}\,.
\]
Then, one can check $\Omega(S(\Gamma,G))=\Omega(\Gamma)\times \Bbb{C}^*$, $LiG(\Omega(S(\Gamma,G)))=3$  and 
$Li(\Omega(S(\Gamma,G)))=\infty$

\subsection{A group With Three Lines in General Position}
Let $a\in \C^*$ and $M_a,B \in SL(3,\C)$ given by:
\[
M_a =
\left (
\begin{array}{lll}
a & 0 &0\\
0 &a &0\\
0 &0 &a^{-2}\\
\end{array}
\right );\,
B =
\left (
\begin{array}{lll}
0 &0& 1\\
1 &0& 0\\
0 &1& 0\\
\end{array}
\right )
\]
 Let $\Gamma_a\subset PSL(3,\Bbb{C})$ be the group generated by $M_a$ and $B$, thus one can check, see \cite{cano}, $\Omega(\Gamma_a)=\C^*\times \C^*$ and  $LiG(\Omega(\Gamma_a))=Li(\Omega(\Gamma_a))=3$.

\subsection{A group With Two Lines in General Position And Infinite Lines}
Let $M\in SL(3,\mathbb{Z})$, also let
 $\alpha,\beta,\overline{\beta}$ with $\alpha>1$,
$\beta\neq \overline{\beta}$ be its eigenvalues. Choose a  real eigenvector
$(a_1,a_2,a_3)$ belonging to $\alpha$ and an eigenvector
$(b_1,b_2,b_3)$ belonging to $\beta$. Now let $G_M$ be the group generated by:
\[
\begin{array}{l}
\gamma_0(w,z)=(\alpha w,\beta z),\\
\gamma_i(w,z)=(w+a_i,z+b_i);\, i=1,2,3.
\end{array}
\]
Then it is easily seen that $\Omega(G_M)=\bigcup_{j=0,1}(\H^{(-1)^{j})}\times \mathbb{C})$, $Li(\Omega(G_M))=\infty$, $LiG(\Omega(G_M))=2$ and $Eq(G_M)=\emptyset$.

\subsection{A group With Two Lines in General Position}
Let $\Gamma\subset PSL(3,\Bbb{C})$ be the group induce by the matrix:
\[
\gamma =
\left (
\begin{array}{lll}
a &0& 0\\
1 &b& 0\\
0 &0& c\\
\end{array}
\right ),
\]
where $\mid a \mid<\mid b\mid<\mid c\mid$ and $abc=1$. Thus it is easily seen, see \cite{Nav2}, that $\Omega(\Gamma)=\Bbb{C}\times \Bbb{C}^*$ and in consequence $LiG(\Omega(\Gamma))=Li(\Omega(\Gamma))=2$.

\subsection{A group With One Line in General Position}
Let $v_1,v_2,v_3,v_4\in \C^{2}$ be points which are $\Bbb{R}$-linearly
independent. Let $g_i$, $i\in\{1,2,3,4\}$, be the translation induced by $v_i$, then
$\Gamma=<g_1,\ldots,g_{4}>$ is a group isomorphic to $\Bbb{Z}^{4}$ which satisfy $\Omega(\Gamma)=\C^2$ and $LiG(\Omega(\Gamma))=1$.\\

The authors are grateful to Professor Jos\'e Seade for his stimulating conversations and encouragement. Part of this research was done while the authors were visiting the
IMATE-UNAM campus Cuernavaca and the FMAT of the UADY. Also,
during this time, the second author was in a postdoctoral year at IMPA,
and they are grateful to these institutions and its people, for their support
and hospitality.

\bibliographystyle{amsplain}

\end{document}